\documentclass[11pt]{article}%
\usepackage[utf8]{inputenc}
\usepackage[T1]{fontenc}
\usepackage[english]{babel}
\usepackage{indentfirst}
\usepackage{latexsym}
\usepackage[colorlinks=true, bookmarksnumbered=true, bookmarksopen=true,
bookmarksopenlevel=3, pdfstartview=FitH, linkcolor=cyan, pdfmenubar=true,
pdftoolbar=true, bookmarks=true,citecolor=cyan, urlcolor=magenta,
filecolor=cyan,menucolor=black,plainpages=false,pdfpagelabels]{hyperref}
\usepackage[lmargin=2cm,tmargin=2cm,bmargin=2cm,rmargin=2cm]{geometry}
\usepackage{amsbsy, amsmath,amsfonts,amssymb,amsthm}
\usepackage{times}
\usepackage{graphicx}
\usepackage{amsmath}
\usepackage{amsfonts}
\usepackage{amssymb}%
\newtheorem{theorem}{Theorem}[section]
\newtheorem{definition}[theorem]{Definición}
\newtheorem{proposition}[theorem]{Proposition}
\newtheorem{corollary}[theorem]{Corollary}
\newtheorem{lemma}[theorem]{Lemma}
\newtheorem{remark}[theorem]{Remark}
\newtheorem{example}[theorem]{Example}
\numberwithin{equation}{section}
\begin{document}

\title{Complemented subspaces of polynomial ideals}
\author{{{Sergio A. Pérez}
{\thanks{S.Pérez was supported by Minciencias, Colombia. (corresponding author) }}}\\{\small  Universidad Pedagógica y Tecnológica de Colombia} 
\\{\small \texttt{Email:Sergio.2060@hotmail.com}}\\\vspace{-0.2cm}}
\maketitle
\begin{abstract}

Given the polynomial ideal $\mathcal{J}\circ\mathcal{P} (^{n}E; F)$, we prove that if $\mathcal{J}\circ\mathcal{P} (^{n}E; F)$ contains an isomorphic copy of $c_{0}$, then $\mathcal{J}\circ\mathcal{P} (^{n}E; F)$ is not complemented in $\mathcal{P} (^{n}E; F)$ for every closed operator ideal $\mathcal{J}\subset \mathcal{L}_{K}$ and every $n\in\mathbb{N}$. Likewise we show that if $\widehat{(\mathcal{J}\circ\mathcal{L})^{fac}}(^{n}E;F)$ contains an isomorphic copy of $c_{0}$, then $\widehat{(\mathcal{J}\circ\mathcal{L})^{fac}}(^{n}E;F)$ is not complemented in $\mathcal{P}(^{n}E; F)$ for every closed operator ideal $\mathcal{J}\subset \mathcal{L}_{K}$ and every $n>1$. When $\mathcal{J}=\mathcal{L}_{K}$, these results generalizes results of several authors
\cite{LEW},\cite{EM},\cite{KALTON},\cite{IOANA},\cite{SERGIO}, among others.



\end{abstract}
\section{Introduction}

The problem of establishing sufficient conditions for the complementation of the subspace of compact linear operators $\mathcal{L}_{K}(E;F)$ in the space $\mathcal{L}(E;F)$ of all continuous linear operators, has been widely studied by many authors. For example, see  Bator and Lewis \cite{LEW}, Emmanuelle \cite{EM}, John \cite{KA}, Kalton \cite{KALTON} and Ghenciu \cite{IOANA}, among others.

Emmanuele \cite{EM} and John \cite{KA} showed that if $c_{0}$ embeds in $\mathcal{L}_{K}(E;F)$ then $\mathcal{L}_{K}(E;F)$ is not complemented in $\mathcal{L}(E;F)$ for every $E$ and $F$ infinite dimensional Banach spaces.

John \cite{KA} proved that if $E$ and $F$ are arbitrary Banach spaces and $T: E\rightarrow F$ is a non compact
operator which admits a factorization $T = A\circ B$ through a Banach space
$G$ with an unconditional basis, then the subspace $\mathcal{L}_{K}(E;F)$
of compact operators contains an isomorphic copy of $c_{0}$ and thus $\mathcal{L}_{K}(E;F)$ is not
complemented in $\mathcal{L}(E;F)$.
John \cite{KA} also proved that if $E$ and $F$ are infinite dimensional Banach spaces, such that each
non compact operator $T \in \mathcal{L}(E;F)$ factors through a Banach space $G$ with an unconditional basis, then the following conditions are equivalent:
\begin{enumerate}
\item $\mathcal{L}_{K} (E; F)=\mathcal{L} (E;F)$.
\item $\mathcal{L} (E; F)$ contains no copy of $\ell_{\infty}$.
\item $\mathcal{L}_{K} (E; F)$ contains no copy of $c_{0}$.
\item  $\mathcal{L}_{K} (E; F)$ is complemented in $\mathcal{L} (E; F)$.
\end{enumerate}

Ghenciu \cite{IOANA} obtained the following result:

\begin{theorem} \label{thm:(Teorema 1009)}(\cite[Theorem 1]{IOANA})
Let $E$ and $F$ be Banach spaces, and let $G$ be a Banach space with an unconditional basis $(g_{n})$ and coordinate functionals $(g^{\prime}_{n})$.

\begin{enumerate}

\item [(a)] If there exist operators $R\in\mathcal{L}(G; F)$ and $S\in\mathcal{L}(E; G)$ such that $(R(g_{n}))$ is a semi-normalized basic sequence
in $F$ and $(S^{\prime}(g^{\prime}_{n}))$ is not relatively compact in $E^{\prime}$, then $\mathcal{L}_{K}(E; F)$ is not complemented in
$\mathcal{L}(E; F)$.

\item [(b)] If there exist operators $R\in\mathcal{L}(G; F)$ and $S\in\mathcal{L}(E; G)$ such that $(R(g_{n}))$ is a semi-normalized basic sequence
in $F$ and $(S^{\prime}(g^{\prime}_{n}))$ is not relatively weakly compact in $E^{\prime}$, then the space of all weakly compact linear operators $\mathcal{L}_{wK}(E; F)$ is not complemented in $\mathcal{L}(E; F)$.
\end{enumerate}
\end{theorem}

This result generalizes results of several authors \cite{EMA},\cite{LEW}, \cite{FEDER}.
Clearly Theorem \ref{thm:(Teorema 1009)} $(a)$ follows from results of Emmanuele  \cite{EM} and John \cite{KA} previously aforementioned.

We consider the space $P_{w}(^{n}E; F)$ of all $n-$ homogeneous polynomials from $E$ into $F$ which are weakly continuous on bounded sets. When $n=1$ we have that $P_{w}(^{n}E; F)=\mathcal{L}_{K}(E; F)$.
\bigskip

Pérez  \cite{SERGIO}  proved the polynomial versions of the previously results, among the most prominent theorems
we find the following:

\begin{theorem} \label{thm:(Teorema 12233)}
Let $E$ and $F$ be Banach spaces, and let $G$ be a Banach space with an unconditional basis $(g_{n})$ and coordinate functionals $(g^{\prime}_{n})$.
If there exist operators $R\in \mathcal{L}(G; F)$ and $S\in\mathcal{L}(E; G)$ such that $(R(g_{n}))$ is a semi-normalized basic sequence
in $F$ and $(S^{\prime}(g^{\prime}_{n}))$ is not relatively compact in $E^{\prime}$, then $\mathcal{P}_{w}(^{n}E; F)$ is not complemented in
$\mathcal{P}(^{n}E; F)$ for every $n\in \mathbb{N}$.
\end{theorem}

\begin{theorem}\label{cor angiev}
Let $E$ be an infinite dimensional Banach space and $n>1$. If $\mathcal{P}_{w} (^{n}E; F)$ contains a copy of $c_{0}$, then $\mathcal{P}_{w} (^{n}E; F)$ is not complemented in $\mathcal{P} (^{n}E; F)$.
\end{theorem}

\begin{theorem}\label{cor olga}
 Let $E$ and $F$ be Banach spaces, with $E$ infinite dimensional, and let $n>1$. If each $P\in \mathcal{P}(^{n}E; F)$ such that $P\notin \mathcal{P}_{w}(^{n}E; F)$ admits a factorization $P=Q\circ T$, where $T\in \mathcal{L}(E;G)$, $Q\in \mathcal{P}(^{n}G;F)$ and $G$ is a Banach space with an unconditional finite dimensional expansion of the identity, then the following conditions are equivalent:
 \begin{enumerate}

\item [(1)] $\mathcal{P}_{w} (^{n}E; F)$ contains a copy of $c_{0}$.
\item [($1^{\prime}$)] $\mathcal{P}_{K} (^{n}E; F)$ contains a copy of $c_{0}$.
\item [(2)] $\mathcal{P}_{w} (^{n}E; F)$ is not complemented in $\mathcal{P} (^{n}E; F)$.
\item [($2^{\prime}$)] $\mathcal{P}_{K} (^{n}E; F)$ is not complemented in $\mathcal{P} (^{n}E; F)$.
\item [(3)] $\mathcal{P}_{w} (^{n}E; F)\neq \mathcal{P} (^{n}E; F)$.
\item [($3^{\prime}$)] $\mathcal{P}_{K} (^{n}E; F)\neq \mathcal{P} (^{n}E; F)$.
\item [(4)] $\mathcal{P}(^{n}E; F)$ contains a copy of $c_{0}$.
\item [(5)] $\mathcal{P}(^{n}E; F)$ contains a copy of $\ell_{\infty}$.
\end{enumerate}
 \end{theorem}

In this paper, we obtain polynomial ideal versions of the preceding results.
\bigskip
\section{Preliminaries}

Let $E$ and $F$ denote Banach spaces over $ \mathbb{K}$, where $ \mathbb{K}$ is $ \mathbb{R}$ or  $\mathbb{C}$. Let $E^{\prime}$
denote the dual of $E$. Let $B_{E}$ denote the unit ball of $E$. By $J_{E}$ we mean the canonical embedding from $E$ to $E^{\prime\prime}$. We denote by $(e_{n})$ the canonical basis of $c_{0}$. Denote by $ \mathcal{L}(E;F)$, $\mathcal{L}_{K}(E;F)$, $\mathcal{F}(E;F)$, $\mathcal{K}_{p}(E;F)$, $\mathcal{QN}_{p}(E;F)$ and $\mathcal{L}_{wK}(E;F)$, respectively, the spaces of all bounded, all compact, all finite range, all $p-$compact, all quasi $p-$nuclear and all weakly compact linear operators from $E$ into $F$. Let $\mathcal{P}(^{n}E; F)$ denote the Banach space of all continuous $n$-homogeneous polynomials from $E$ into $F$. We omit $F$ when $F = \mathbb{K}$. Given a continuous $n-$homogeneous polynomial $P:E\rightarrow F$, by $\check{P}$ we mean the (unique) symmetric $n-$ linear operator  such that
$\check{P}(x,\ldots,x)=P(x)$ for every $x\in E$.
Let $\mathcal{P}_{w}(^{n}E; F)$ denote the subspace of all $P\in\mathcal{P}(^{n}E; F)$
which are weakly continuous on bounded sets, that is
the restriction $P|_{B}: B \rightarrow F $ is continuous for each bounded set $B\subset E$,
when $B $ and $F$ are endowed with the weak topology and the norm topology,
respectively. Let $\mathcal{P}_{K}(^{n}E; F)$ denote the
subspace of all $P\in \mathcal{P}(^{n}E; F)$  which map bounded sets onto relatively compact sets.
Let $\mathcal{P}_{wK}(^{n}E; F)$ denote the
subspace of all $P\in \mathcal{P}(^{n}E; F)$  which map bounded sets onto relatively weakly compact sets.

We always have the inclusions
$$ \mathcal{P}_{w}(^{n}E; F)\subset \mathcal{P}_{K}(^{n}E; F)\subset\mathcal{P}_{wK}(^{n}E; F)\subset \mathcal{P}(^{n}E; F),$$
as can be derived from results in \cite{AR}.
We refer to  \cite{MUJICA} for background information on the theory of polynomials on Banach spaces.
 \bigskip

By $E_{1}\hat{\otimes}_{\pi}\ldots\hat{\otimes}_{\pi}E_{n}$ we denote the completed projective tensor product of
$E_{1}, \ldots,E_{n}$. If $E_{1} = \ldots= E_{n} = E$ we write $\hat{\otimes}_{n,\pi}E$. By $\hat{\otimes}_{n,s,\pi}E$ we denote the
$n-$fold completed symmetric tensor product of $E$.
\bigskip

$E$ is isomorphic to a complemented subspace of $F$ if and only if there are $A \in
\mathcal{L}(E; F)$ and $B\in \mathcal{L}(F;E)$ such that $B\circ A = I$. $E$ is said to have an unconditional finite dimensional expansion of the identity if there is a sequence of bounded linear operators $A_{n}:E\rightarrow E$ of finite rank, such that for $x\in E$
$$\sum_{n=1}^{\infty}A_{n}(x)=x$$
unconditionally.

We will say that the series $\displaystyle\sum_{n=1}^{\infty}x_{n}$ of elements of $E$ is weakly unconditionally Cauchy if $\displaystyle\sum_{n=1}^{\infty}|x^{\prime}(x_{n})|<\infty$ for all $x^{\prime}\in E^{\prime}$ or, equivalently if
$$\sup\bigg\{\bigg\|\sum_{n\in F}x_{n}\bigg\| ; F\subset \mathbb{N}, F finite\bigg\}<\infty.$$

A sequence $(x_{n})\subset E$ is a semi-normalized basic sequence if $(x_{n})$ is a Schauder basis for the closed subspace $M =
\overline{[x_{n} : n\in \mathbb{N}]}$, and moreover there are constant $a$ and $b$ such that $0< a<\|x_{n}\|<b$ for all $n\in\mathbb{N}$.
We denote by $(e_{n})$ the canonical basis of $c_{0}$. If $\Sigma$ is an algebra of subsets of a set $\Omega$, then a finitely additive vector
measure $\mu:\Sigma\rightarrow E$ is said to be strongly additive if the series $\displaystyle\sum_{n=1}^{\infty}\mu(A_{n})$ converges in norm for each
sequence $(A_{n})$ of pairwise disjoint members of $\Sigma$.
The Diestel-Faires Theorem (see \cite[p. 20, Theorem 2]{DIESTEL}) asserts that if $\Sigma$ is a $\sigma-$ algebra and there is a measure $\mu:\Sigma\rightarrow E$ which is not strongly additive, then $E$ contains an isomorphic copy of $\ell_{\infty}$.
\begin{remark}
The proof of the Diestel-Faires Theorem \cite[p. 20, Theorem 2]{DIESTEL} does not ensure the existence of a topological isomorphism, indeed,
only can be proved that there exists a copy of $c_{0}$ and $\ell_{\infty}$ in $E$, respectively.
\end{remark}

In \cite{PELLEGRINO} is defined an ideal of \textit{ multilinear mappings} (or \textit{multi-ideal}) $\mathcal{M}$ as a subclass of the
class of all continuous multilinear mappings between Banach spaces such that
for all $n\in\mathbb{N}$ and Banach spaces $E_{1},\ldots,E_{n}$ and $F$ the components $\mathcal{M}(E_{1},\ldots,E_{n};F)=\mathcal{L}(E_{1},\ldots,E_{n};F)\cap \mathcal{M}$ satisfy:

\begin{enumerate}

\item [(i)] $\mathcal{M}(E_{1},\ldots,E_{n};F)$ is a linear subspace of $\mathcal{L}(E_{1},\ldots,E_{n};F)$  which contains
the $n-$linear mappings of finite type.

\item [(ii)] The ideal property: if $A\in \mathcal{M}(E_{1},\ldots,E_{n};F)$, $u_{j}\in \mathcal{L}(G_{j},E_{j})$ for $j=1,\ldots,n$ and $t\in \mathcal{L}(F;H)$, then $t\circ A\circ (u_{1},\ldots,u_{n}) \in \mathcal{M}(G_{1},\ldots,G_{n};H)$.
\end{enumerate}
If $\|·\|_{\mathcal{M}}:\mathcal{M}\rightarrow \mathbb{R}^{+}$  satisfies

\begin{enumerate}
\item [($i^{\prime}$)] $(\mathcal{M}(E_{1},\ldots,E_{n};F),\|\cdot\|_{\mathcal{M}})$  is a normed (Banach) space for all Banach spaces
$E_{1},\ldots,E_{n}$, $F$ and all $n\in\mathbb{N}$.
\item [($ii^{\prime}$)] $\|A^{n}:\mathbb{K}^{n}\rightarrow \mathbb{K} : A^{n}(x_{1},\ldots,x_{n})=x_{1}\ldots x_{n}\|_{\mathcal{M}}=1$ for all $n\in\mathbb{N}$.
\item [($iii^{\prime}$)] Si $A\in \mathcal{M}(E_{1},\ldots,E_{n};F)$, $u_{j}\in \mathcal{L}(G_{j},E_{j})$ for $j=1,\ldots,n$ and $t\in \mathcal{L}(F;H)$, then $\|t\circ A\circ (u_{1},\ldots,u_{n})\|_{\mathcal{M}}\leq \|t\|\|A\|_{\mathcal{M}}\|u_{1}\|\ldots\|u_{n}\|$,
    then $(\mathcal{M},\|\cdot\|_{\mathcal{M}})$ is called a normed (Banach) multi-ideal.

The multi-ideal $\mathcal{M}$ is said to be closed if each $\mathcal{M}(E_{1},\ldots,E_{n};F)$ is a closed subspace of $\mathcal{L}(E_{1},\ldots,E_{n};F)$.
\end{enumerate}
\bigskip

An ideal of \textit{homogeneous polynomials} (or \textit{polynomial ideal}) $\mathcal{Q}$ is a subclass
of the class of all continuous homogeneous polynomials between Banach
spaces such that for all $n\in\mathbb{N}$ and Banach spaces $E$ and $F$, the components
$\mathcal{Q}(^{n}E;F)=\mathcal{P}(^{n}E;F)\cap \mathcal{Q}$ satisfy:
\begin{enumerate}

\item [(i)] $\mathcal{Q}(^{n}E;F)$ is a linear subspace of $\mathcal{P}(^{n}E;F)$ which contains the $n-$homogeneous polynomials of finite type.
\item [(ii)] The ideal property: if $u\in\mathcal{L}(G;E)$, $P\in \mathcal{Q}(^{n}E;F)$ and $t\in\mathcal{L}(F,H)$, then
the composition $t\circ P\circ u$ is in $\mathcal{Q}(^{n}G;H)$.

\end{enumerate}

When  $n=1$ $\mathcal{Q}$ is called an \textbf{operator ideal}.

\begin{definition}\label{corolario 1}
Let $\mathcal{I}$ be an operator ideal.
\begin{enumerate}
\item [(a)] An application $n-$ linear $A\in\mathcal{L}(E_{1},E_{2},\ldots, E_{n}; F)$ belongs to $\mathcal{I}\circ \mathcal{L}$, in this case
we write $A\in \mathcal{I}\circ \mathcal{L}(E_{1},E_{2},\ldots, E_{n}; F)$, if there are a Banach space $G$, an $n-$ linear mapping $B\in\mathcal{L}(E_{1},E_{2},\ldots, E_{n}; G)$ and an operator $u\in \mathcal{I}(G;F)$, such that $A=u\circ B$.

\item [(b)] An $n-$homogeneous polynomial $P\in \mathcal{P}(^{n}E; F)$ belongs to $\mathcal{I}\circ \mathcal{P}$ - in this case
we write $P\in \mathcal{I}\circ \mathcal{P}(^{n}E; F)$ - if there are a Banach space $G$, an $n-$homogeneous
polynomial $Q\in \mathcal{P}(^{n}E;G)$ and an operator $u\in \mathcal{I}(G; F)$ such that $P = u \circ Q$.
\end{enumerate}
\end{definition}

\begin{example}
Let $\mathcal{K}$ and $\mathcal{W}$ be the closed operator ideals formed by all compact and
weakly compact linear operators, respectively. By $\mathcal{P}_{\mathcal{K}}$ and $\mathcal{P}_{\mathcal{W}}$ we mean the
classes of all compact and weakly compact polynomials, respectively. The
equalities $\mathcal{P}_{\mathcal{K}} = \mathcal{K}\circ\mathcal{P}$ and $\mathcal{P}_{\mathcal{W}}=\mathcal{W}\circ\mathcal{P}$ were proved by R. Ryan \cite{RYAN}.
\end{example}

\begin{definition}
Given an operator ideal $\mathcal{I}$ and Banach spaces $E$ and $F$, we define $\mathcal{I}^{dual}(E;F)=\{T\in \mathcal{L}(E;F) : T^{\prime}\in\mathcal{I}(F^{\prime}; E^{\prime})\}$.

\end{definition}

\begin{definition}
Given a Banach space $E$ and an operator ideal $\mathcal{I}$. The collection of all $\mathcal{I}$ bounded sets of $E$ is defined by $\mathcal{C}_{I}(E)=\{A\subset E: \exists \ T\in\mathcal{I}(Z;E) \ \textit{for some Banach  space } \ $Z$\  \textit{such that}\  A\subset T(B_{Z})\}$.

If $\mathcal{I}=\mathcal{K}$ or $\mathcal{I}=\mathcal{W}$ the compact and weakly compact operator ideals, respectively, then
$\mathcal{C}_{I}(E)$ are the relatively compact and relatively weakly compact sets of $E$, respectively.
\end{definition}

\bigskip
In \cite{GERALDO}, given a multi-ideal $\mathcal{M}$, is constructed an ideal $\mathcal{M}^{fac}$ such that every $n-$ linear operator $A$ in $\mathcal{M}^{fac}$ is strongly $\mathcal{M}-$ factorable in the sense  that it factors through multilinear operator in $\mathcal{M}$ with respect
to any partition of the set $\{1,\ldots,n\}$. For $\mathcal{L}_{K}=\mathcal{K}\circ\mathcal{L}$ (compact multilinear operators) we have that
$\mathcal{L}_{K}^{fac}=\mathcal{L}_{wb}$ (weakly continuous on bounded sets multilinear operators).

\section{The main results}

The next result is a generalization of theorem  \ref{thm:(Teorema 1009)}. The proof of the next theorem is based in ideas of \cite[Theorem 6(i)]{LEWIS}.

\begin{theorem} \label{thm:(Teorema 10)}
Let $E$ and $F$ be Banach spaces, and let $G$ be a Banach space with an unconditional basis $(g_{n})$ and coordinate functionals $(g^{\prime}_{n})$.
\begin{enumerate}

\item [(a)] If there exist operators $R\in\mathcal{L}(G; F)$ and $S\in\mathcal{L}(E; G)$ such that $(R(g_{n}))$ is a semi-normalized basic sequence
in $F$ and $(S^{\prime}(g^{\prime}_{n_{k}}))\notin \mathcal{C}_{\mathcal{I}}(E^{\prime})$ for all subsequence $(n_{k})\subset \mathbb{N}$, where $\mathcal{I}$ and $\mathcal{J}$ be closed operator ideals such that $\mathcal{J}\subset \mathcal{I}^{dual}$. Then $\mathcal{J}\circ\mathcal{L}(E; F)$ is not complemented in
$\mathcal{L}(E; F)$.

\item [(b)] If $E$ is a reflexive Banach space and there exist operators $R\in\mathcal{L}(G; F)$ and $S\in\mathcal{L}(E; G)$ such that $(R(g_{n}))$ is a semi-normalized basic sequence in $F$ and  $(S^{\prime}(g^{\prime}_{n_{k}}))\notin \mathcal{C}_{\mathcal{I}}(E^{\prime})$ for all subsequence $(n_{k})\subset \mathbb{N}$, where $\mathcal{I}$ and $\mathcal{J}$ be closed operator ideals such that $\mathcal{J}^{dual}\subset \mathcal{I}$. Then $\mathcal{J}\circ\mathcal{L}(E; F)$ is not complemented in $\mathcal{L}(E; F)$.
\end{enumerate}
\end{theorem}

\begin{proof}

$(a)$ Let $(P_{A})$ be the family of projections associated with $(g_{n})$, $R$ and $S$ as in the hypothesis, and set $\mu(A)=R\circ P_{A}\circ S$, $A\subseteq \mathbb{N}$. Let $E_{0}$ be a separable subspace of $E$ such that $\|x^{\prime}\|=\|x^{\prime}|_{E_{0}}\|$ for all $x^{\prime}\in [S^{\prime}(g^{\prime}_{n}): n\geq 1]$. Let $(y^{\prime}_{n})$ be the sequence of biorthogonal coefficients corresponding to $(R(g_{n}))$, and $(f^{\prime}_{n})$ be the sequence of Hahn-Banach extension to $F^{\prime}$. Now suppose that there exist a projection $\pi:\mathcal{L}(E; F)\rightarrow \mathcal{J}\circ\mathcal{L}(E; F)$. Let $J:F\rightarrow \ell_{\infty}$ be an operator that is an isometry on $[R(g_{n}): n\geq 1]$.
We define $\upsilon:\mathcal{P}(\mathbb{N})\rightarrow \mathcal{L}(E_{0}; \ell_{\infty})$ by $\upsilon(A)=(J\circ \pi\circ \mu(A)- J\circ \mu(A))|_{E_{0}}$,
$A\subseteq \mathbb{N}$. Since $\mu(\{n\})$ is an operator of finite range, then $\upsilon(\{n\})=0$ for every $n\in\mathbb{N}$. Apply \cite[Lemma 4 ]{LEWIS} to obtain an infinite subset $M$ of $\mathbb{N}$ so that $J\circ \pi\circ \mu(M)=J\circ \mu(M)$ on $E_{0}$. Therefore $(J\circ \mu(M))|_{E_{0}}\in \mathcal{J}(E_{0};\ell_{\infty})$. Since $J$ is an isometry on $[R(g_{n}): n\geq 1]$ and $ \mu(M)(E_{0})\subset [R(g_{n}): n\geq 1]$, we have that $\mu(M)|_{E_{0}}\in \mathcal{J}(E_{0};F)$. As $\mathcal{J}\subset \mathcal{I}^{dual}$ it implies that $(\mu(M)|_{E_{0}})^{\prime}\in \mathcal{I}(F^{\prime}; E^{\prime}_{0})$. On the other hand, there exist a constant $c>0$ such that $\|f^{\prime}_{n}\|\leq c$ for all $n\in \mathbb{N}$. Is easy to see that
$c(\mu(M)|_{E_{0}})^{\prime}\big(\frac{f^{\prime}_{n}}{c}\big)=(S^{\prime}(g^{\prime}_{n}))|_{E_{0}}$ for all $n\in M$. Consider the operator
$\psi:[S^{\prime}(g^{\prime}_{n})|_{E_{0}}: n\in M]\rightarrow [S^{\prime}(g^{\prime}_{n}): n\in M]$, defined by $\psi(S^{\prime}(g^{\prime}_{n})|_{E_{0}})=S^{\prime}(g^{\prime}_{n})$. Since $\|x^{\prime}\|=\|x^{\prime}|_{E_{0}}\|$ for all $x^{\prime}\in [S^{\prime}(g^{\prime}_{n}): n\geq 1]$, then $\psi$ is an operator linear and continuo, thus $\psi\circ c(\mu(M)|_{E_{0}})^{\prime})\in \mathcal{I}(F^{\prime};E^{\prime})$ and $(S^{\prime}(g^{\prime}_{n}))_{n\in M}\subset \psi\circ c(\mu(M)|_{E_{0}})^{\prime})(B_{F^{\prime}})$. Thus $(S^{\prime}(g^{\prime}_{n}))_{n\in M}\in \mathcal{C}_{\mathcal{I}}(E^{\prime})$, but it is a contradiction with the hypothese. Therefore $\mathcal{J}\circ\mathcal{L}(E; F)$ is not complemented in $\mathcal{L}(E; F)$.
\bigskip

$(b)$ If we asssume that exist a projetion $\pi:\mathcal{L}(E; F)\rightarrow \mathcal{J}\circ\mathcal{L}(E; F)$, then
analogously as in $(i)$, we have that $\mu(M)|_{E_{0}}\in \mathcal{J}(E_{0};F)$ for any infinite subset $M$ of $\mathbb{N}$.
Is easy to see that $(\mu(M)|_{E_{0}})^{\prime\prime}\circ J_{E_{0}}=J_{F}\circ \mu(M)|_{E_{0}}\in \mathcal{J}(E_{0};F^{\prime\prime})$. Since $E$ is reflexive then $E_{0}$ is also reflexive, therefore $(\mu(M)|_{E_{0}})^{\prime\prime}=(\mu(M)|_{E_{0}})^{\prime\prime}\circ J_{E_{0}}\circ J^{-1}_{E_{0}}\in \mathcal{J}(E^{\prime\prime}_{0};F^{\prime\prime})$. By hypothesis $\mathcal{J}^{dual}\subset \mathcal{I}$, thus $(\mu(M)|_{E_{0}})^{\prime}\in \mathcal{I}(F^{\prime}; E^{\prime}_{0})$. As in the proof of $(i)$ it implies that $(S^{\prime}(g^{\prime}_{n}))_{n\in M}\in \mathcal{C}_{\mathcal{I}}(E^{\prime})$, but it is absurd. Finally we obtain the desired result.
\end{proof}
\bigskip

\begin{remark}
By \cite[Corollary 3.4 ]{DELGADO} and \cite[Proposition 3.8 ]{DELGADO}, $\mathcal{QN}_{p}=\mathcal{K}_{p}^{dual}$ and $\mathcal{K}_{p}=\mathcal{QN}_{p}^{dual}$, respectively. Thus, if we take $\mathcal{J}=\mathcal{QN}_{p}$, $\mathcal{I}=\mathcal{K}_{p}$ and $\mathcal{J}=\mathcal{K}_{p}$, $\mathcal{I}=\mathcal{QN}_{p}$, then
 the previously theorem can be applied.


In \cite[Example 4.4 ]{Geraldo} we can to see more examples of known ideals of operators that satisfy the theorem previously mentioned.

\end{remark}

\begin{theorem} \label{thm:(Teorema 1111)}
Let $E$ and $F$ be Banach spaces, and let $G$ be a Banach space with an unconditional basis $(g_{n})$ and coordinate functionals $(g^{\prime}_{n})$.

\begin{enumerate}

\item [(a)] If there exist operators $R\in\mathcal{L}(G; F)$ and $S\in\mathcal{L}(E; G)$ such that $(R(g_{n}))$ is a semi-normalized basic sequence
in $F$ and $(S^{\prime}(g^{\prime}_{n_{k}}))\notin \mathcal{C}_{\mathcal{I}}(E^{\prime})$ for all subsequence $(n_{k})\subset \mathbb{N}$, where $\mathcal{I}$ and $\mathcal{J}$ be closed operator ideals such that $\mathcal{J}\subset \mathcal{I}^{dual}$.
Then $\mathcal{J}\circ\mathcal{P}(^{n}E; F)$ is not complemented in
$\mathcal{P}(^{n}E; F)$ for every $n\in \mathbb{N}$.

\item [(b)] If $E$ is a reflexive Banach space and there exist operators $R\in\mathcal{L}(G; F)$ and $S\in\mathcal{L}(E; G)$ such that $(R(g_{n}))$ is a semi-normalized basic sequence in $F$ and  $(S^{\prime}(g^{\prime}_{n_{k}}))\notin \mathcal{C}_{\mathcal{I}}(E^{\prime})$ for all subsequence $(n_{k})\subset \mathbb{N}$, where $\mathcal{I}$ and $\mathcal{J}$ be closed operator ideals such that $\mathcal{J}^{dual}\subset \mathcal{I}$. Then $\mathcal{J}\circ\mathcal{P}(^{n}E; F)$ is not complemented in $\mathcal{P}(^{n}E; F)$ for every $n\in \mathbb{N}$.
\end{enumerate}
\end{theorem}

\begin{proof}

$(a)$ The case $n=1$ follows from Theorem \ref{thm:(Teorema 10)} $(a)$. Suppose that exists a projection $\pi:\mathcal{P}(^{n}E; F)\rightarrow \mathcal{J}\circ\mathcal{P}(^{n}E; F)$ for some $n>1$. By a result of Ryan  \cite{RYAN} there exists an isomorphism
$$P\in\mathcal{P}(^{n}E; F)\rightarrow \check{P}\in\mathcal{L}(\hat{\otimes}_{n,s,\pi}E; F).$$ Using \cite[Proposition 3.2 ]{PELLEGRINO} we have that
$\mathcal{J}\circ\mathcal{P}(^{n}E; F)$ is isomorphic to $\mathcal{J}(\hat{\otimes}_{n,s,\pi}E; F)$, then there exists a projection
$\rho: \mathcal{L}(\hat{\otimes}_{n,s,\pi}E; F)\rightarrow \mathcal{J}(\hat{\otimes}_{n,s,\pi}E; F)$. By a result of Blasco \cite[Theorem 3]{BLASCO} $E$ is isomorphic to a complemented subspace of $\hat{\otimes}_{n,s,\pi}E$. Hence there exist operators $A\in\mathcal{L}(E;\hat{\otimes}_{n,s,\pi}E)$ and $B\in\mathcal{L}(\hat{\otimes}_{n,s,\pi}E; E)$ such that $B\circ A=I$. Consider the operator
$$\varphi: T\in\mathcal{L}(E;F)\rightarrow \rho(T\circ B)\circ A\in \mathcal{J}(E; F).$$

If $T\in\mathcal{J}(E; F)$, then $T\circ B\in\mathcal{J}(\hat{\otimes}_{n,s,\pi}E; F)$ and therefore $\rho(T\circ B)\circ A=T\circ B\circ A=T$.
Thus $\varphi: T\in\mathcal{L}(E;F)\rightarrow \rho(T\circ B)\circ A\in \mathcal{J}(E; F)$ is a projection, contradicting the case $n=1$.

$(b)$  The proof of $(b)$ is almost identical to the proof of $(a)$.
\end{proof}

When $\mathcal{J}=\mathcal{I}=\mathcal{L}_{K}$  Theorem \ref{thm:(Teorema 1111)} correspond to
part $(a)$ of \cite[Theorem 3.1 ]{SERGIO}, and part (b) of \cite[Theorem 3.1 ]{SERGIO} is obtained when $\mathcal{J}=\mathcal{I}=\mathcal{L}_{wK}$.

\begin{proposition} \label{thm:(Teorema 111)}
Let $T\in (\mathcal{I}\circ \mathcal{L})^{fac}(^{n}E;F)$ and $\varphi\in E^{\prime}$. Then $\varphi T\in(\mathcal{I}\circ \mathcal{L})^{fac}(^{n+1}E;F)$,
where $$\varphi T(x_{1},\ldots,x_{n+1})=\varphi(x_{n+1})T(x_{1},\ldots,x_{n}).$$
\end{proposition}

\begin{proof}
Let $\pi=\{j^{1}_{1},\ldots,j^{1}_{k_{1}}\}\cup\ldots\cup\{j^{m}_{1},\ldots,j^{m}_{k_{m}}\}$ be a partition of $\{1,\ldots,n+1\}$.
\bigskip

Case $1$: Suppose that $\{n+1\}\in \pi$. Let $\widehat{\pi}=\pi-\{n+1\}$ be a partition of $\{1,\ldots,n\}$. We can write
$$T\stackrel{\widehat{\pi}}{=}C\circ (B_{1},\ldots,B_{m-1}),$$

where $B_{i}\in \mathcal{I}\circ \mathcal{L}(\underbrace{E\times \ldots \times E}_{k_{i} \ \   times}; G_{i}),$ $i=1,\ldots, m-1,$
and $C\in \mathcal{L}(G_{1}\times\ldots \times G_{m-1}; F)$. We consider $G_{m}=\mathbb{K}$ and $B_{m}=\varphi$.
Define $\widehat{C}\in \mathcal{L}(G_{1}\times\ldots \times G_{m}; F)$ by

$$\widehat{C}(g_{1},\ldots,g_{m})=C(g_{1},\ldots,g_{m-1})g_{m},$$

for all $g_{i}\in G_{i}$, $i=1,\ldots,m$.

We can see easily that

$$\varphi T(x_{1},\ldots,x_{n+1})=\widehat{C}(B_{1}(x_{j_{1}}^{1},\ldots,x_{j_{k_{1}}}^{1}),\ldots,B_{m-1}(x_{j_{1}}^{m-1},\ldots,x_{j_{k_{m-1}}}^{m-1}),B_{m}(x_{n+1})).$$

Therefore $\varphi T\in(\mathcal{I}\circ \mathcal{L})^{fac}(^{n+1}E;F)$.

\bigskip

Case $2$: Suppose that $n+1\in \{j^{l}_{1},\ldots,j^{l}_{k_{l}}\}$ for some $1\leq l\leq m$, where $k_{l}>1$.
Let $\widehat{\pi}=\pi-\{n+1\}$ be a partition of $\{1,\ldots,n\}$. Then
$$T\stackrel{\widehat{\pi}}{=}C\circ (B_{1},\ldots,B_{m}),$$
where $B_{l}\in \mathcal{I}\circ \mathcal{L}(\underbrace{E\times \ldots \times E}_{k_{l}-1 \ \   times}; G_{l}),$  $B_{i}\in \mathcal{I}\circ \mathcal{L}(\underbrace{E\times \ldots \times E}_{k_{i} \ \   times}; G_{i}),$  $i=1,\ldots l-1,l+1,\ldots,m$,
and $C\in \mathcal{L}(G_{1}\times\ldots \times G_{m}; F)$. We define
$\widehat{B_{l}}\in \mathcal{I}\circ \mathcal{L}(\underbrace{E\times \ldots \times E}_{k_{l} \ \   times}; G_{l}),$

by

$$\widehat{B_{l}}(x_{1},\ldots,x_{k_{l}})=\varphi(x_{k_{l}})B_{l}(x_{1},\ldots,x_{k_{l}-1}).$$

Finally, we have that

$$\varphi T(x_{1},\ldots,x_{n+1})=\varphi(x_{n+1})T(x_{1},\ldots,x_{n})=C(B_{1}(x_{j_{1}}^{1},\ldots,x_{j_{k_{1}}}^{1}),\ldots,\widehat{B_{l}}(x_{j_{1}}^{l},\ldots,x_{n+1}),\ldots, B_{m}(x_{j_{1}}^{m},\ldots,x_{j_{k_{m}}}^{m})).$$

This complete the proof.

\end{proof}

\begin{proposition} \label{thm:(Teorema 112)}
Let $T\in (\mathcal{I}\circ \mathcal{L})^{fac}(^{n}E;F)$ and $e\in E$. Then $\overline{T_{j}} \in(\mathcal{I}\circ \mathcal{L})^{fac}(^{n-j}E;F)$,
where $\overline{T_{j}}(x_{1},\ldots,x_{n-j})=T(x_{1},\ldots,x_{n-j},\underbrace{e,\ldots,e}_{j \ times})$, for every $j=1,2,...,n-1$.
\end{proposition}

\begin{proof}
To prove the proposition by induction on $j$ it suffices to prove for the case $j=1$.
Let $\pi=\{j^{1}_{1},\ldots,j^{1}_{k_{1}}\}\cup\ldots\cup\{j^{m}_{1},\ldots,j^{m}_{k_{m}}\}$ be a partition of $\{1,\ldots,n-1\}$.
Consider the partition $\widehat{\pi}=\pi\cup\{n\}$ of $\{1,\ldots,n\}$. By hipothesis $T\in (\mathcal{I}\circ \mathcal{L})^{fac}(^{n}E;F)$, then we can write

$$T\stackrel{\widehat{\pi}}{=}C\circ (B_{1},\ldots,B_{m+1}),$$

where $B_{i}\in \mathcal{I}\circ \mathcal{L}(\underbrace{E\times \ldots \times E}_{k_{i} \ \   times}; G_{i}),$ $i=1,\ldots, m+1,$
and $C\in \mathcal{L}(G_{1}\times\ldots \times G_{m+1}; F)$.
Let $\widehat{C}\in \mathcal{L}(G_{1}\times\ldots \times G_{m}; F)$ define by

$$\widehat{C}(g_{1},\ldots,g_{m})=C(g_{1},\ldots,g_{m},B_{m+1}(e)),$$

then
\begin{eqnarray}
\overline{T_{1}}(x_{1},\ldots,x_{n-1})=T(x_{1},\ldots,x_{n-1},e)=C(B_{1}(x_{j_{1}}^{1},\ldots,x_{j_{k_{1}}}^{1}),\ldots,B_{m}(x_{j_{1}}^{m},\ldots, x_{j_{k_{m}}}^{m}),B_{m+1}(e))\nonumber\\
=\widehat{C}(B_{1}(x_{j_{1}}^{1},\ldots,x_{j_{k_{1}}}^{1}),\ldots,B_{m}(x_{j_{1}}^{m},\ldots, x_{j_{k_{m}}}^{m}))\ \ \ \ \ \ \ \ \ \ \ \ \ \ \ \ \ \nonumber
\end{eqnarray}

Thus $\overline{T_{1}} \in(\mathcal{I}\circ \mathcal{L})^{fac}(^{n-1}E;F)$. This complete the proof.
\end{proof}

\begin{theorem} \label{thm:(Teorema 11)}
Let $E$ and $F$ be Banach spaces, and let $G$ be a Banach space with an unconditional basis $(g_{n})$ and coordinate functionals $(g^{\prime}_{n})$.
 If there exist operators $R\in\mathcal{L}(G; F)$ and $S\in\mathcal{L}(E; G)$ such that $(R(g_{n}))$ is a semi-normalized basic sequence
in $F$ and $(S^{\prime}(g^{\prime}_{n_{k}}))\notin \mathcal{C}_{\mathcal{I}}(E^{\prime})$ for all subsequence $(n_{k})\subset \mathbb{N}$, where $\mathcal{I}$ and $\mathcal{J}$ be closed operator ideals such that $\mathcal{J}\subset \mathcal{I}^{dual}$. Then $\widehat{(\mathcal{J}\circ\mathcal{L})^{fac}}(^{n}E;F)$ is not complemented in
$\mathcal{P}(^{n}E; F)$ for every $n\in \mathbb{N}$.
\end{theorem}
\begin{proof}
When $n=1$ we have that $\widehat{(\mathcal{J}\circ\mathcal{L})^{fac}}(^{n}E;F)=\mathcal{J}\circ\mathcal{L}(E;F)$ and $\mathcal{P}(^{n}E; F)=\mathcal{L}(E; F)$. Thus the case $n=1$ follows from Theorem \ref{thm:(Teorema 10)} $(a)$. To prove the theorem by induction on $n$ it suffices to prove that
if $\widehat{(\mathcal{J}\circ\mathcal{L})^{fac}}(^{n+1}E;F)$ is complemented in $\mathcal{P}(^{n+1}E; F)$, then $\widehat{(\mathcal{J}\circ\mathcal{L})^{fac}}(^{n}E;F)$ is complemented in $\mathcal{P}(^{n}E; F)$. Aron and Schottenloher \cite[Proposition 5.3]{AR} proved that $\mathcal{P}(^{n}E; F)$ is isomorphic to a complemented subspace of
$\mathcal{P}(^{n+1}E; F)$ when $F$ is the scalar field, but their proof works equally well when $F$ is an arbitrary Banach space (see \cite[Proposition 5]{BLASCO}). Thus there
exist operators $A\in\mathcal{L}(\mathcal{P}(^{n}E; F); \mathcal{P}(^{n+1}E; F))$ and $B\in\mathcal{L}(\mathcal{P}(^{n+1}E; F);\mathcal{P}(^{n}E; F))$
such that $B\circ A=I$. The operator $A$ is of the form
$$A(P)(x)=\varphi_{0}(x)P(x)$$
for every $P\in\mathcal{P}(^{n}E; F)$ and $x\in E$, where $\varphi_{0}\in E^{\prime}$ verifies that $\|\varphi_{0}\|=1=\varphi_{0}(x_{0})$, where
$x_{0}\in E$ and $\|x_{0}\|=1$. Using Proposition \ref{thm:(Teorema 111)} we have that if $P\in \widehat{(\mathcal{J}\circ\mathcal{L})^{fac}}(^{n}E;F)$ then $A(P)\in \widehat{(\mathcal{J}\circ\mathcal{L})^{fac}}(^{n+1}E;F)$.  On the other hand, the operator $B$ is of the form $B=A^{-1}\circ D$, where $D: \mathcal{P}(^{n+1}E; F)\rightarrow \mathcal{P}(^{n+1}E; F)$ is defined by $D(P)(x)=P(x)- P(x-\varphi_{0}(x)x_{0})$ for every $P\in\mathcal{P}(^{n+1}E; F)$ and $x\in E$. More exactly $B(P)$ is given by

$$B(P)(x)=\sum_{j=1}^{n+1}\binom{n+1}{j}(-1)^{j+1}\varphi_{0}^{j-1}(x)\check{P}(\underbrace{x,\ldots,x}_{n+1-j \ times},\underbrace{e,\ldots,e}_{j \ times})$$

for all $x\in E$, (see \cite[p. 597]{CALIS}). Applying simultaneously Propositions \ref{thm:(Teorema 111)} and \ref{thm:(Teorema 112)} we can to see that if $P\in \widehat{(\mathcal{J}\circ\mathcal{L})^{fac}}(^{n+1}E;F)$ then $B(P)\in \widehat{(\mathcal{J}\circ\mathcal{L})^{fac}}(^{n}E;F)$.
\bigskip

Let us assume that $\widehat{(\mathcal{J}\circ\mathcal{L})^{fac}}(^{n+1}E;F)$ is complemented in $\mathcal{P}(^{n+1}E; F)$, and let $\pi: \mathcal{P}(^{n+1}E; F)\rightarrow \widehat{(\mathcal{J}\circ\mathcal{L})^{fac}}(^{n+1}E;F)$ be a projection. Consider the operator
$$\rho=B\circ \pi\circ A: \mathcal{P}(^{n}E; F)\rightarrow \widehat{(\mathcal{J}\circ\mathcal{L})^{fac}}(^{n}E;F).$$
If $P\in \widehat{(\mathcal{J}\circ\mathcal{L})^{fac}}(^{n}E;F)$, then $A(P)\in \widehat{(\mathcal{J}\circ\mathcal{L})^{fac}}(^{n+1}E;F)$, and therefore
$$\rho(P)=B\circ \pi\circ A(P)=B\circ A(P)=P.$$
Thus $\rho:\mathcal{P}(^{n}E; F)\rightarrow \widehat{(\mathcal{J}\circ\mathcal{L})^{fac}}(^{n}E;F)$ is a projection, and therefore $\widehat{(\mathcal{J}\circ\mathcal{L})^{fac}}(^{n}E;F)$ is complemented
in $\mathcal{P}(^{n}E; F)$. This completes the proof.

\end{proof}

Ghenciu \cite{IOANA} derived as corollaries of Theorem \ref{thm:(Teorema 1009)} results of several authors \cite{LEW},\cite{EMA}, \cite{FEDER}, \cite{KALTON} and \cite{KA}. We now apply Theorems \ref{thm:(Teorema 1111)} and \ref{thm:(Teorema 11)} to obtain versions about ideals of polynomials of those corollaries.

\begin{corollary} \label{thm:(Teorema 200)}
If $c_{0}\hookrightarrow F$ and $E^{\prime}$ contains
a weak star-null  sequence $(x^{\prime}_{n})$ such that $(x^{\prime}_{n_{k}})\notin \mathcal{C}_{\mathcal{I}}(E^{\prime})$ for all subsequence $(n_{k})\subset \mathbb{N}$, where $\mathcal{I}$ and $\mathcal{J}$ be closed operator ideal such that $\mathcal{J}\subset \mathcal{I}^{dual}$. Then $\mathcal{J}\circ \mathcal{P}(^{n}E; F)$ is not complemented in
$\mathcal{P}(^{n}E; F)$ for every $n\in \mathbb{N}$.
\end{corollary}

\begin{corollary} \label{thm:(Teorema 201)}
If $c_{0}\hookrightarrow F$ and $E$ contains a complemented copy of $c_{0}$. Then $\mathcal{J}\circ \mathcal{P}(^{n}E; F)$ is not complemented in $\mathcal{P}(^{n}E; F)$ for every $n\in \mathbb{N}$, where $\mathcal{J}\subset \mathcal{L}_{wK}$ is a closed operator ideal.
\end{corollary}


\begin{corollary} \label{thm:(Teorema 202)}
If $F$ contains a copy of $\ell_{1}$ and $\mathcal{L}(E;\ell_{1})\neq \mathcal{L}_{K}(E;\ell_{1})$. Then $\mathcal{J}\circ \mathcal{P}(^{n}E; F)$ is not complemented in $\mathcal{P}(^{n}E; F)$ for every $n\in \mathbb{N}$, where $\mathcal{J}\subset \mathcal{L}_{wK}$ is a closed operator ideal .
\end{corollary}

When $n=1$ and $\mathcal{J}=\mathcal{L}_{wK}$, Corollaries \ref{thm:(Teorema 200)}, \ref{thm:(Teorema 201)} and \ref{thm:(Teorema 202)} correspond to \cite[Corollaries 2,3 and 5]{IOANA}.
Ghenciu derived those corollaries by observing that $E$ and $F$ satisfy the hypothesis of Theorem  \ref{thm:(Teorema 1009)} $(b)$. Since the hypothesis
of Theorem \ref{thm:(Teorema 1009)} $(b)$ implies  the hypothesis of Theorem \ref{thm:(Teorema 1111)} $(a)$, we see that Corollaries \ref{thm:(Teorema 200)}, \ref{thm:(Teorema 201)} and \ref{thm:(Teorema 202)} follow from Theorem \ref{thm:(Teorema 1111)} $(a)$.

\begin{corollary}\label{cor 33}
 If $F$ contains a copy of $c_{0}$ and $E$ is infinite dimensional, then for all closed operator ideal $\mathcal{J}\subset \mathcal{L}_{K}$ we have that:
\begin{enumerate}
\item [(a)] $\mathcal{J}\circ \mathcal{P}(^{n}E; F)$ is not complemented in $\mathcal{P}(^{n}E; F)$ for every $n\in \mathbb{N}$.

\item [(b)] $\widehat{(\mathcal{J}\circ\mathcal{L})^{fac}}(^{n}E;F)$ is not complemented in
$\mathcal{P}(^{n}E; F)$ for every $n\in \mathbb{N}$.
\end{enumerate}
\end{corollary}


\begin{corollary}\label{cor 34}
 If $E$ contains a complemented copy of $\ell_{1}$ and $F$ is infinite dimensional, then for all closed operator ideal $\mathcal{J}\subset \mathcal{L}_{K}$ we have that:
\begin{enumerate}
\item [(a)] $\mathcal{J}\circ \mathcal{P}(^{n}E; F)$ is not complemented in $\mathcal{P}(^{n}E; F)$ for every $n\in \mathbb{N}$.

\item [(b)] $\widehat{(\mathcal{J}\circ\mathcal{L})^{fac}}(^{n}E;F)$ is not complemented in
$\mathcal{P}(^{n}E; F)$ for every $n\in \mathbb{N}$.
\end{enumerate}
\end{corollary}

When $n=1$ and $\mathcal{J}=\mathcal{L}_{K}$, Corollaries \ref{cor 33} and \ref{cor 34} correspond to \cite[Corollaries 4 and 6]{IOANA}.
Ghenciu derived those corollaries by observing that $E$ and $F$ satisfy the hypothesis of Theorem \ref{thm:(Teorema 1009)} $(a)$. Since the hypothesis
of Theorem \ref{thm:(Teorema 1009)} $(a)$ implies the hypothesis of Theorems \ref{thm:(Teorema 1111)} $(a)$ and \ref{thm:(Teorema 11)}, we see that Corollaries \ref{cor 33} and \ref{cor 34} follow from Theorems \ref{thm:(Teorema 1111)} $(a)$ and \ref{thm:(Teorema 11)}.

\begin{corollary}\label{cor 35}
 If $E$ contains a copy of $\ell_{1}$ and $F$ contains a copy of $\ell_{p}$, with $2\leq p<\infty$, then for all closed operator ideal $\mathcal{J}\subset \mathcal{L}_{K}$ we have that:
\begin{enumerate}
\item [(a)] $\mathcal{J}\circ \mathcal{P}(^{n}E; F)$ is not complemented in $\mathcal{P}(^{n}E; F)$ for every $n\in \mathbb{N}$.

\item [(b)] $\widehat{(\mathcal{J}\circ\mathcal{L})^{fac}}(^{n}E;F)$ is not complemented in
$\mathcal{P}(^{n}E; F)$ for every $n\in \mathbb{N}$.
\end{enumerate}
\end{corollary}

When $\mathcal{J}=\mathcal{L}_{K}$, Corollary \ref{cor 35} correspond to \cite[Corollary 3.8]{SERGIO}. Since the hypothesis
of \cite[Corollary 3.8]{SERGIO} implies the hypothesis of Theorems \ref{thm:(Teorema 1111)} $(a)$ and \ref{thm:(Teorema 11)}, we see that Corollary \ref{cor 35} follow from Theorems \ref{thm:(Teorema 1111)} $(a)$ and \ref{thm:(Teorema 11)}.

The next Corollary is a generalization of \cite[Theorem 2]{EM} and \cite[Theorem 2]{KA}. The proof is based in ideas of \cite[Corollary 11]{LEWIS}.

\begin{corollary}\label{cor 36}
Let $E$ and $F$ be infinite dimensional Banach spaces. If $c_{0}\hookrightarrow\mathcal{J}(E;F)$, then $\mathcal{J}(E;F)$ is not complemented in $\mathcal{L}(E;F)$ for all closed operator ideal $\mathcal{J}\subset \mathcal{L}_{K}$.
\end{corollary}

\begin{proof}
 By Corollaries \ref{cor 33} and \ref{cor 34} we may suppose without loss of generality that $F$ contains no copy of $c_{0}$ and $E$ contains no complemented copy of $\ell_{1}$. Thus, by \cite[Theorem 4]{KALTON} $\mathcal{J}(E;F)$ contains no copy of $\ell_{\infty}$. If $T_{n}$ is a copy of the unit vector basis of $c_{0}$ in $\mathcal{J}(E;F)$. Then $\displaystyle\sum_{n=1}^{\infty}T_{n}(x)$ converges unconditionally for each $x\in E$. Define $\mu:\wp(\mathbb{N})\rightarrow \mathcal{J} (E; F)$ by $\mu(A)(x)=\displaystyle\sum_{n\in A}T_{n}(x)$, $x\in E$, $A\subset \mathbb{N}$. Suppose there is a projection $\pi:\mathcal{L}(E;F)\rightarrow \mathcal{J}(E;F)$. Then $\pi(T_{n})=T_{n}$ for every $n\in \mathbb{N}$. If $(\|T_{n}\|)$ does not converge to zero, we can apply the Diestel-Faires Theorem \cite[p. 20, Theorem 2]{DIESTEL} to the measure $\pi\circ \mu$ and obtain $\ell_{\infty}\hookrightarrow \mathcal{J}(E;F)$. Therefore $\|T_{n}\|\rightarrow 0$, but this is absurd too, because $(T_{n})$ is a copy of unit vectors. This complete the proof.
\end{proof}

\begin{proposition}\label{cor 37}
Let $E$ and $F$ be infinite dimensional Banach spaces. If $c_{0}\hookrightarrow\mathcal{J}\circ \mathcal{P}(^{n}E;F)$, then $\mathcal{J}\circ \mathcal{P}(^{n}E;F)$ is not complemented in $\mathcal{P}(^{n}E;F)$ for every $n\in\mathbb{N}$, and every closed ideal $\mathcal{J}\subset \mathcal{L}_{K}$.
\end{proposition}
\begin{proof}
By \cite[Proposition 3.2]{PELLEGRINO} we have that $\mathcal{J}\circ \mathcal{P}(^{n}E;F)$ is isomorphic to $\mathcal{J}(\hat{\otimes}_{n,s,\pi}E; F)$. Thus
the result follows from Corollary \ref{cor 36}.
\end{proof}

The proof of \cite[Lemma 2]{KALTON} can be applied to get the following lemma.

\begin{lemma}\label{cor 3900}
Assume $E$ is separable, $\widehat{(\mathcal{J}\circ\mathcal{L})^{fac}}(^{k}E;F)$ is complemented in $\mathcal{P}(^{k}E;F)$, and an linear bounded operator
$\phi:\ell_{\infty}\rightarrow \mathcal{P}(^{k}E;F)$ is given with the following properties:
\begin{enumerate}
\item [(a)] $\phi(e_{n})\in \widehat{(\mathcal{J}\circ\mathcal{L})^{fac}}(^{k}E;F)$ for all $n\in \mathbb{N}$.
\item [(b)] $\{\phi(\xi)(x): \  \ \xi\in \ell_{\infty}, x\in E\}\subset F $ is separable.
\end{enumerate}

Then, for every infinite subset $M\subset \mathbb{N}$, there exists an infinite subset $M_{0}\subset M$ with $\phi(\xi)\in \widehat{(\mathcal{J}\circ\mathcal{L})^{fac}}(^{k}E;F)$ for all $\xi\in\ell_{\infty}(M_{0})$.
\end{lemma}

\begin{lemma}\label{cor 3901}
Suppose $E$ contains a complemented copy of $\ell_{1}$. Then $\widehat{(\mathcal{J}\circ\mathcal{L})^{fac}}(^{n}E;F)$ is not complemented in
$\mathcal{P}(^{n}E;F)$ for every $F$ and $n>1$, where $\mathcal{J}$ is a closed operator ideal $\mathcal{J}\subset \mathcal{L}_{K}$.
\end{lemma}
\begin{proof}
As in \cite[Lemma 5]{GONZA}, we can reduce the problem to the case $E=\ell_{1}$.

Fix $v\in B_{F}$ and define the operator $\phi:\ell_{\infty}\rightarrow \mathcal{P}(^{k}\ell_{1};F)$ by $\phi(\xi)(x)=\displaystyle\sum_{i=1}^{\infty}\xi_{i}x_{i}^{k}v$ for $\xi=(\xi_{i})\in \ell_{\infty}$, $x=(x_{i})\in \ell_{1}$ and $k>1$.
Note that $\phi(e_{n})$ are polynomials of finite type, so $\phi(e_{n})\in \widehat{(\mathcal{J}\circ\mathcal{L})^{fac}}(^{k}\ell_{1};F)$ for all $n\in\mathbb{N}$. Suppose there exists a projection $\pi:\mathcal{P}(^{k}\ell_{1};F)\rightarrow \widehat{(\mathcal{J}\circ\mathcal{L})^{fac}}(^{k}\ell_{1};F)$ for any  $k>1$. Using Lemma
\ref{cor 3900}, there is an infinite subset $M\subset \mathbb{N}$ such that $\phi(\xi)\in \widehat{(\mathcal{J}\circ\mathcal{L})^{fac}}(^{k}\ell_{1};F)$ for all $\xi\in\ell_{\infty}(M)$. But it is absurd because $\phi(\xi)\in \mathcal{P}_{wb}(^{k}\ell_{1};F)$ if and only if $\xi\in c_{0}$ (to see \cite[Lemma 5]{GONZA}). Therefore $\widehat{(\mathcal{J}\circ\mathcal{L})^{fac}}(^{n}E;F)$ is not complemented in
$\mathcal{P}(^{n}E;F)$ for every $F$ and $n>1$.
\end{proof}

\begin{proposition}\label{cor 38}
Let $E$  be an infinite dimensional Banach space and $n>1$. If $c_{0}\hookrightarrow\widehat{(\mathcal{J}\circ\mathcal{L})^{fac}}(^{n}E;F)$, then $\widehat{(\mathcal{J}\circ\mathcal{L})^{fac}}(^{n}E;F)$ is not complemented in $\mathcal{P}(^{n}E;F)$ for every closed operator ideal $\mathcal{J}\subset \mathcal{L}_{K}$.
\end{proposition}
\begin{proof}
By Corollary \ref{cor 33} and Lemma \ref{cor 3901} we way suppose that $F$  contains no copy of $c_{0}$ and $E$ contains no complemented copy of $\ell_{1}$. By \cite[Theorem 3]{GONZA} $\widehat{(\mathcal{J}\circ\mathcal{L})^{fac}}(^{n}E;F)$ contains no copy of $\ell_{\infty}$. Let $(P_{i})$ be a copy of the unit vector basis $(e_{i})$ of $c_{0}$ in $\widehat{(\mathcal{J}\circ\mathcal{L})^{fac}}(^{n}E;F)$. Then $$\sup\bigg\{\bigg\|\sum_{i\in F}e_{i}\bigg\| ; F\subset\mathbb{N}, F finite\bigg\}=1.$$ By a result of Bessaga and Pelczynski \cite{BES} (see also \cite[p. 44, Theorem 6]{DIESTELL}) the series $\displaystyle\sum_{i=1}^{\infty} e_{i}$ is weakly unconditionally Cauchy in $c_{0}$. This implies that the series
 $\displaystyle\sum_{i=1}^{\infty} P_{i}$ is weakly unconditionally Cauchy in $\widehat{(\mathcal{J}\circ\mathcal{L})^{fac}}(^{n}E;F)$. For every $\varphi\in F^{\prime}$ and $x\in E$ we consider the continuous linear functional $$\psi: P\in\widehat{(\mathcal{J}\circ\mathcal{L})^{fac}}(^{n}E;F)\rightarrow \varphi(P(x))\in \mathbb{C}.$$ Since the series
 $\displaystyle\sum_{i=1}^{\infty} P_{i}$ is weakly unconditionally Cauchy in $\widehat{(\mathcal{J}\circ\mathcal{L})^{fac}}(^{n}E;F)$,
 $\displaystyle\sum_{i=1}^{\infty}|\psi( P_{i})|=\displaystyle\sum_{i=1}^{\infty}|\varphi( P_{i}(x))|<\infty$ for every $\varphi\in F^{\prime}$ and $x\in E$. This shows that
 $\displaystyle\sum_{i=1}^{\infty} P_{i}(x)$ is weakly unconditionally Cauchy in $F$ for each $x\in E$. Finally since $F$ contains no copy of $c_{0}$, an
 application of \cite[p. 45, Theorem 8 ]{DIESTELL} shows that $\displaystyle\sum_{i=1}^{\infty} P_{i}(x)$ converges unconditionally in $F$ for each $x\in E$. Let $\mu: \wp(\mathbb{N})\rightarrow \mathcal{P} (^{n}E; F)$ be the finitely additive vector measure defined by $\mu(A)(x)=\displaystyle\sum_{i\in A}P_{i}(x)$ for each $x\in E$ and $A\subset \mathbb{N}$.
Suppose there is a projection $\pi:\mathcal{P}(^{n}E; F)\rightarrow \widehat{(\mathcal{J}\circ\mathcal{L})^{fac}}(^{n}E;F)$ for any $n\in\mathbb{N}$. Then $\pi(P_{i})=P_{i}$ for each $i\in \mathbb{N}$.
 If the sequence $(\|P_{i}\|)$ does not converge to zero, then there is $\epsilon>0$ and a subsequence $(i_{k})$ of $\mathbb{N}$, such that $\|P_{i_{k}}\|>\epsilon$ for each $k\in \mathbb{N}$. But this implies that the measure $\pi\circ\mu:\wp(\mathbb{N})\rightarrow \widehat{(\mathcal{J}\circ\mathcal{L})^{fac}}(^{n}E;F)$ is not strongly additive. Then the Diestel-Faires Theorem \cite[p. 20, Theorem 2]{DIESTEL} would imply that $\widehat{(\mathcal{J}\circ\mathcal{L})^{fac}}(^{n}E;F)$ contains a copy of $\ell_{\infty}$. Therefore $\|P_{i}\|\rightarrow 0$, but this is absurd too, because $(P_{i})$ is a copy of $(e_{i})$. This complete the proof.

\end{proof}

When $\mathcal{J}=\mathcal{L}_{K}$ Proposition \ref{cor 38} gives \cite[Proposition 3.10]{SERGIO}.

\begin{theorem} \label{thm:(Teorema 1797)}
Let $E$ and $F$ be Banach spaces and $P\in \mathcal{P}(^{n}E; F)$ such that $P\notin \widehat{(\mathcal{J}\circ\mathcal{L})^{fac}}(^{n}E;F)$, where $\mathcal{J}\subset \mathcal{L}_{K}$ is a closed operator ideal. Suppose that $P$ admits a factorization $P=Q\circ T$ through a Banach space $G$ with an unconditional finite dimensional expansion of the identity, where $T\in \mathcal{L}(E;G)$ and $Q\in \mathcal{P}(^{n}G;F)$. Then $\widehat{(\mathcal{J}\circ\mathcal{L})^{fac}}(^{n}E;F)$ contains a copy of $c_{0}$ and thus $\widehat{(\mathcal{J}\circ\mathcal{L})^{fac}}(^{n}E;F)$ is not complemented in $\mathcal{P}(^{n}E;F)$.
\end{theorem}
\begin{proof}

There exists a sequence $(A_{i})\in \mathcal{F}(G;G)$ such that, for every $g\in G$, we have $g=\sum_{i=1}^{\infty}A_{i}(g)$ unconditionally. Then,

\begin{equation*}
\begin{split}
Q\bigg(\displaystyle\sum_{i=1}^{k}A_{i}(g)\bigg)&=\displaystyle\sum_{i_{1},\ldots,i_{n}=1}^{k}\check{Q}(A_{i_{1}}(g),\ldots,A_{i_{n}}(g))\\
&=\displaystyle\sum_{m=1}^{k}\bigg(\displaystyle\sum_{\max\{i_{1},\ldots,i_{n}\}=m}\check{Q}(A_{i_{1}}(g),\ldots,A_{i_{n}}(g))\bigg)\\
&=\displaystyle\sum_{m=1}^{k}Q_{m}(g).
\end{split}
\end{equation*}
Since $Q$ is continuous, we have that $Q(g)=\displaystyle\sum_{m=1}^{\infty}Q_{m}(g)$ for all $g\in G$.

To prove that $Q_{m}\in \widehat{(\mathcal{J}\circ\mathcal{L})^{fac}}(^{n}G;F)$ for every $m\in\mathbb{N}$, it suffices to prove that the multilinear map $\varphi\in \mathcal{L}(^{n}G;F)$ defined by,

$\varphi(g_{1},\ldots,g_{n})=\check{Q}(C_{1}(g_{1}),\ldots,C_{n}(g_{n}))$ is contained in $(\mathcal{J}\circ\mathcal{L})^{fac}(^{n}G;F)$,
for any $C_{1},\ldots,C_{n}\in \mathcal{F}(G;G)$. Let $\pi=\{j_{1}^{1},\ldots,j_{k_{1}}^{1}\}\cup\ldots\cup \{j_{1}^{m},\ldots,j_{k_{m}}^{m}\}$ be a partition of
$\{1,\ldots,n\}$. We consider $B_{i}:\underbrace{G\times\ldots\times G}_{k_{i} \ \   times}\rightarrow \hat{\otimes}_{k_{i},s,\pi} G$, $i=1,\ldots,m$, defined by
$$B_{i}(g_{j_{1}^{i}},\ldots,g_{j_{k_{i}}^{i}})=C_{j_{1}^{i}}(g_{j_{1}^{i}})\widehat{\otimes}_{s,\pi}\ldots\widehat{\otimes}_{s,\pi}C_{j_{k_{i}}^{i}}(g_{j_{k_{i}}^{i}}),$$
for all $g_{j_{1}^{1}},\ldots,g_{j_{k_{m}}^{m}}\in G$.

 $B_{i}=\widehat{B}_{i}\circ \psi$, where $\psi:\underbrace{G\times\ldots\times G}_{k_{i} \ \   times}\rightarrow \hat{\otimes}_{k_{i},s,\pi} G$ is the canonical isomorphism, and
$\widehat{B}_{i}:\hat{\otimes}_{k_{i},s,\pi} G\rightarrow \hat{\otimes}_{k_{i},s,\pi} G$ is defined by

$\widehat{B}_{i}(g_{j_{1}^{i}}\hat{\otimes}_{\pi}\ldots\hat{\otimes}_{\pi}g_{j_{k_{i}}^{i}})=C_{j_{1}^{i}}(g_{j_{1}^{i}})\widehat{\otimes}_{\pi}\ldots\widehat{\otimes}_{\pi}C_{j_{k_{i}}^{i}}(g_{j_{k_{i}}^{i}}).$
Like $C_{j_{1}^{i}}(G),\ldots,C_{j_{k_{i}}^{i}}(G)$ are spaces finite dimensional, then $\widehat{B}_{i}\in \mathcal{J}(\hat{\otimes}_{k_{i},\pi} G;\hat{\otimes}_{k_{i},\pi} G)$ for every $i=1,\ldots,m$. Therefore $B_{i}\in \mathcal{J}\circ\mathcal{L}(\underbrace{G\times\ldots\times G}_{k_{i} \ \   times};\hat{\otimes}_{k_{i},\pi} G)$, $i=1,\ldots,m$. Let $C\in\mathcal{L}(\hat{\otimes}_{k_{1},\pi} G,\ldots,\hat{\otimes}_{k_{m},\pi} G;F)$ given by

$$C(g_{j_{1}^{1}}\hat{\otimes}_{\pi}\ldots\hat{\otimes}_{\pi}g_{j_{k_{1}^{1}}},\ldots,g_{j_{1}^{m}}\hat{\otimes}_{\pi}\ldots\hat{\otimes}_{\pi}g_{j_{k_{m}}^{m}})=
\check{Q}(g_{j_{1}^{1}},\ldots,g_{j_{k_{1}^{1}}},\ldots,g_{j_{1}^{m}},\ldots,g_{j_{k_{m}}^{m}}).$$

We have that

$$\varphi(g_{1},\ldots,g_{n})=\check{Q}(C_{1}(g_{1}),\ldots,C_{n}(g_{n}))=\check{Q}(C_{j_{1}^{1}}(g_{j_{1}^{1}}),\ldots,C_{j_{k_{m}}^{m}}(g_{j_{k_{m}}^{m}}))=C(B_{1}(g_{1}^{1},\ldots,g_{k_{1}}^{1}),\ldots,B_{m}(g_{j_{1}^{m}},\ldots,g_{j_{k_{m}}^{m}}))$$

for all $g_{1},\ldots,g_{n}\in G$.

Therefore, $Q_{m}\in\widehat{(\mathcal{J}\circ\mathcal{L})^{fac}}(^{n}G;F)$ for every $m\in\mathbb{N}$, and so,
$Q_{m}\circ T\in\widehat{(\mathcal{J}\circ\mathcal{L})^{fac}}(^{n}E;F)$ for every $m\in\mathbb{N}$.

Choosing finite subsets $I_{1},\ldots,I_{k}$ of integers, we have

$$\bigg\|\sum_{i_{1}\in I_{1},\ldots,i_{k}\in I_{k}}\check{Q}(A_{i_{1}}(g),\ldots,A_{i_{k}}(g))\bigg\|=\bigg\|\check{Q}\bigg(\sum_{i_{1}\in I_{1}}A_{i_{1}}(g),\ldots,\sum_{i_{k}\in I_{k}}A_{i_{k}}(g)\bigg)\bigg\|
\leq\|\check{Q}\|\bigg\|\sum_{i_{1}\in I_{1}}A_{i_{1}}(g)\bigg\|\ldots\bigg\|\sum_{i_{k}\in I_{k}}A_{i_{k}}(g)\bigg\|$$
Hence, the series

$$\sum_{i_{1}\in I_{1},\ldots,i_{k}\in I_{k}}^{\infty}\check{Q}(A_{i_{1}}(g),\ldots,A_{i_{k}}(g))$$
is unconditionally convergent for all $g\in G$ \cite[Theorem 1.9]{DJT}. Therefore, $P(x)=\displaystyle\sum_{m=1}^{\infty}Q_{m}(T(x))$ unconditionally.
Moreover, by the uniform boundedness principle \cite[Theorem 2.6]{MUJICA}, we have

$$\sup\bigg\{\bigg\|\sum_{m\in F}Q_{m}\circ T\bigg\| ; F\subset \mathbb{N},\ F \ finite\bigg\}<\infty.$$
So $\displaystyle\sum_{m=1}^{\infty} Q_{m}\circ T$ is weakly unconditionally
Cauchy in $\widehat{(\mathcal{J}\circ\mathcal{L})^{fac}}(^{n}E;F)$. Since $P\notin \widehat{(\mathcal{J}\circ\mathcal{L})^{fac}}(^{n}E;F)$, an application
of \cite[p.45, Theorem 8]{DIESTELL} shows that $\widehat{(\mathcal{J}\circ\mathcal{L})^{fac}}(^{n}E;F)$ contains a copy of $c_{0}$, and therefore by Corollary \ref{cor 36} and Proposition \ref{cor 38} $\widehat{(\mathcal{J}\circ\mathcal{L})^{fac}}(^{n}E;F)$ is not complemented in $\mathcal{P}(^{n}E;F)$.
\end{proof}

\begin{corollary}\label{cor 3000000}
 Let $E$ and $F$ be Banach spaces, with $E$ infinite dimensional, and let $n>1$. If each $P\in \mathcal{P}(^{n}E; F)$ such that $P\notin \widehat{(\mathcal{J}\circ\mathcal{L})^{fac}}(^{n}E;F)$ admits a factorization $P=Q\circ T$, where $T\in \mathcal{L}(E;G)$, $Q\in \mathcal{P}(^{n}G;F)$ and $G$ is a Banach space with an unconditional finite dimensional expansion of the identity, then the following conditions are equivalent for any closed operator ideal $\mathcal{J}\subset\mathcal{L}_{K}$.
 \begin{enumerate}

\item [(1)] $\widehat{(\mathcal{J}\circ\mathcal{L})^{fac}}(^{n}E;F)$ contains a copy of $c_{0}$.
\item [(2)] $\widehat{(\mathcal{J}\circ\mathcal{L})^{fac}}(^{n}E;F)$ is not complemented in $\mathcal{P} (^{n}E; F)$.
\item [(3)] $\widehat{(\mathcal{J}\circ\mathcal{L})^{fac}}(^{n}E;F)\neq \mathcal{P} (^{n}E; F)$.
\item [(4)] $\mathcal{P}(^{n}E; F)$ contains a copy of $c_{0}$.
\item [(5)] $\mathcal{P}(^{n}E; F)$ contains a copy of $\ell_{\infty}$.
\end{enumerate}
 \end{corollary}
 \begin{proof}

 $(1)\Rightarrow (2)$ by Proposition \ref{cor 38}.


$(2)\Rightarrow (3)$ is obvious.

$(3)\Rightarrow (1)$ by Theorem \ref{thm:(Teorema 1797)}.

$(1)\Rightarrow (4)$ is obvious.

$(4)\Rightarrow (3)$ suppose $(4)$ holds and $(3)$ does not hold. Then $\widehat{(\mathcal{J}\circ\mathcal{L})^{fac}}(^{n}E;F)=\mathcal{P}(^{n}E; F)\supset c_{0}$. Thus $(1)$ holds,
and therefore $(3)$ holds, a contradiction.

$(5)\Rightarrow (4)$ is obvious.

$(4)\Rightarrow (5)$ Since $(4)\Rightarrow (1)$ $\widehat{(\mathcal{J}\circ\mathcal{L})^{fac}}(^{n}E;F)\subset\mathcal{P}_{K} (^{n}E; F)$ contains a copy of $c_{0}$. By a result of Ryan \cite{RYAN} $\mathcal{P}(^{n}E; F)$ and $\mathcal{P}_{K} (^{n}E; F)$ are isometrically isomorphic to $\mathcal{L}(\widehat{\otimes}_{n,s,\pi}E; F)$ and $\mathcal{L}_{K}(\widehat{\otimes}_{n,s,\pi}E; F)$, respectively. Thus $\mathcal{L}_{K}(\widehat{\otimes}_{n,s,\pi}E; F)$ contains a copy of $c_{0}$.
Since $E$ is infinite dimensional, $\widehat{\otimes}_{n,s,\pi}E$ is also infinite dimensional.Then by combining the proofs of \cite[Theorem 6, $(iii)\Rightarrow(ii)$]{KALTON} and \cite[Remark 3 e) $2\Rightarrow 3$ ]{KA} we can conclude that $\mathcal{L}(\widehat{\otimes}_{n,s,\pi}E; F)$ contains a copy of $\ell_{\infty}$ and the result follows.

Thus $(1)$, $(2)$, $(3)$, $(4)$ and $(5)$ are equivalent.
\bigskip

In particular if $\mathcal{J}=\mathcal{L}_{K}$ and $E$ has an unconditional finite dimensional expansion of the identity we obtain \cite[Theorem 7]{GONZA}. The assumptions of this corollary apply also if $F$ is a complemented subspace of a space with an unconditional basis.

\end{proof}

\end{document}